\documentclass[psamsfonts]{amsart}

\usepackage{amssymb,amsfonts,amscd}
\usepackage[utf8]{inputenc}
\usepackage[colorlinks,linktocpage]{hyperref}
\hypersetup{linkcolor=[rgb]{0,0,0.715}}
\hypersetup{citecolor=[rgb]{0,0.715,0}}

\newtheorem{theorem}{Theorem}[section]
\newtheorem{corollary}[theorem]{Corollary}

\newtheorem{lemma}[theorem]{Lemma}

\newtheorem*{thm1}{Main Theorem}

\theoremstyle{definition}
\newtheorem{defn}[theorem]{Definition}

\theoremstyle{remark}
\newtheorem{remark}[theorem]{Remark}

\makeatletter
\let\c@equation\c@theorem
\makeatother
\numberwithin{equation}{section}

\title[]{$\mathrm{RCD}^{*}(K,N)$ spaces are semi-locally simply connected}

\author[]{Jikang Wang}
\address[Jikang Wang]{Fields Institute for Research in Mathematical Sciences, Toronto, ON, Canada}
\email{jikangwang1117@gmail.com}
\thanks{The author is supported by Fields Institute for Research in Mathematical Sciences.}
\thanks{}

\thanks{}

\DeclareUnicodeCharacter{03C3}{\ensuremath{\sigma}}
\DeclareUnicodeCharacter{2217}{\ensuremath{*}}
\begin{document}
\date{}
\maketitle
\begin{abstract}
It was shown by Mondino-Wei in \cite{MondinoWei2019} that any $\mathrm{RCD}^{*}(K,N)$ space $(X,d,\mathfrak{m})$ has a universal cover. We prove that for any point $x \in X$ and $R>0$, there exists $r<R$ such that any loop in $B_r(x)$ is contractible in $B_R(x)$; in particular, $X$ is semi-locally simply connected and the universal cover of $X$ is simply connected. This generalizes earlier work in \cite{Wang2021} that any Ricci limit space is semi-locally simply connected.  
\end{abstract}
\section{Introduction}
We are interested in local topology of non-smooth length metric spaces. Assume that $(M_i,p_i)$ is a sequence of $n$-dim Riemannian manifolds with $\mathrm{Ric} \ge -(n-1)$. By Gromov's pre-compactness theorem, passing to a subsequence if necessary, $(M_i,p_i)$ converges to $(Y,p)$ in the Gromov-Hausdorff sense. We call all such limit spaces $Y$ Ricci limit spaces. 

The geometric structure of Ricci limit spaces were well studied by Cheeger-Colding-Naber \cite{CheegerColding1997, CheegerColding2000a,CheegerColding2000b,CheegerNaber2013,CheegerNaber2015, ColdingNaber2012}. The first topological result about Ricci limit spaces was shown by Sormani-Wei that any Ricci limit space has a universal cover. Recently the author proved in \cite{Wang2021} that any Ricci limit space is semi-locally simply connected; see also \cite{PanWang2021,PanWei2019} for the cases with volume conditions. Recall that we say a metric space $Y$ is semi-locally simply connected if for any $y \in Y$, there exists $r>0$ such that any loop in $B_r(y)$ is contractible in $Y$.

In this paper we consider  $\mathrm{RCD}^{*}(K,N)$ spaces which generalize Ricci limit spaces, see section 2.1 for further references about $\mathrm{RCD}^{*}(K,N)$ spaces. Mondino-Wei proved that any $\mathrm{RCD}^{*}(K,N)$ space $X$ has a universal cover \cite{MondinoWei2019} while it was unknown whether the universal cover is simply connected. We shall prove that $X$ is semi-locally simply connected; in particular, the universal cover of $X$ is simply connected.     
\begin{thm1}
Assume that a measured metric space $(X,d,\mathfrak{m})$ is an $\mathrm{RCD}^{*}(K,N)$ space for some $K \in \mathbb{R}$ and $N \in (1, \infty)$. Then for any $x \in X$ and $R>0$, there exists $r>0$ so that any loop in $B_r(x)$ is contractible in $B_R(x)$. In particular, $X$ is semi-locally simply connected.
\end{thm1}

\begin{corollary}(cf. \cite{MondinoWei2019})
In the setting of the main theorem, the universal cover of $X$ is simply connected. The revised fundamental group defined in \cite{MondinoWei2019} (deck transformations on the universal cover) is isomorphic to $\pi_1(X)$.
\end{corollary}

We should mention that Santos-Rodríguez and Zamora recently proved many properties of the fundamental group of an $\mathrm{RCD}^{*}(K,N)$ space \cite{SantosZamora}.

\begin{corollary}(cf. \cite{Tuschman1995,Wang2021})
Assume that a sequence of measured metric spaces $(X_i,d_i,\mathfrak{m}_i)$ are $\mathrm{RCD}^{*}(K,N)$ space for some $K \in \mathbb{R}$ and $N \in (1, \infty)$ and Diam$(X_i) \le 1$. Suppose that $(X_i,d_i,\mathfrak{m}_i)$ measured Gromov-Hausdorff converges to $(X,d,\mathfrak{m})$. For $i$ large enough, there is an onto homomorphism $\phi_i : \pi_1(X_i) \to \pi_1(X)$. 
\end{corollary}

The proof of the main theorem is similar to the proof in \cite{Wang2021} for the Ricci limit space case: we need to show a weak homotopy control property, then use such control to construct a homotopy map. We briefly discuss the proof and point out some differences compared with \cite{Wang2021}.

Since we study local relative fundamental group, it's natural to consider local covers (cover of a ball) and corresponding deck transformations. For a Ricci limit space case $(M_i,y_i) \overset{GH}\longrightarrow (Y,y)$, we can consider universal cover of balls in $M_i$, saying $\widetilde{B_4(y_i)}$. We have equivariant Gromov-Hausdorff convergence. Then we show that there is a slice $S$ in \cite{PanWang2021}, which implies a weak homotopy control property: for any loop in a small and fixed ball of $y$, we can find a nearby loop in $M_i$ which is homotopic to a loop contained in a very small ball with controlled homotopy image. Here "a very small ball" means the radius converges to $0$ as $i \to \infty$. If we observe $M_i$ from the limit space $Y$, a very small ball in $M_i$ has no difference between a point because of the GH-distance gap between $Y$ and $M_i$. Using the weak homotopy control, we can construct a homotopy map on $Y$ and prove that any Ricci limit space is semi-locally simply connected. 

For an $\mathrm{RCD}^{*}(K,N)$ space $(X,d,\mathfrak{m})$, however, it's unknown whether there is a sequence of manifolds converging to $(X,d,\mathfrak{m})$. Therefore, for any $p \in X$, we shall consider relative $\delta$-cover (see section 2.2) of a ball in $X$, saying $\tilde{B}(p,4,40)^{\delta}$.

Using Theorem \ref{stability}, $\tilde{B}(p,4,40)^{\delta}$ are all same for all $\delta$ small enough. In particular, it implies a weak homotopy control property on $X$: for any loop $\gamma$ in a small neighborhood of $p$, $\gamma$ is homotopic to some loops, each of which is contained in a $\delta$-ball, with controlled homotopy image. The reason that we have "some loops" instead of "one loop" is that the fundamental group of $\tilde{B}(p,4,40)^{\delta}$ is not trivial but generated by loops in $\delta$-balls. As we said before, there is no difference between a very small ball and a point during the construction. We may choose $\delta$ arbitrarily small while $r$ is fixed due to Theorem \ref{stability}.  Therefore we can construct a homotopy map and prove that any loop in a small neighborhood of $p$ is contractible with controlled homotopy image.  

The author would thank Xingyu Zhu for helpful discussions about $\mathrm{RCD}^*$ spaces and Jiayin Pan, Jaime Santos-Rodr{\'i}guez, Sergio Zamora-Barrera for some comments to simplify the proof.
\section{Preliminaries}

\subsection{$\mathrm{RCD}^*(K,N)$ spaces}

We will consider a geodesic metric space $(X,d)$ with a $\sigma$-finite Borel positive measure $\mathfrak{m}$; the triple
$(X,d,\mathfrak{m})$ will be called a metric measure space.

$\mathrm{RCD}^*(K,N)$ spaces are the Ricci curvature analog of the celebrated
Alexandrov spaces. $\mathrm{RCD}^*(K,N)$ spaces are the generalization of Riemannian manifolds with the volume measure, Ricci curvature bounded from below by $K$ and dimension bounded above by $N$. 
We do not include the definition of $\mathrm{RCD}^*(K,N)$ spaces since it is technical and will not be used explicitly; see \cite{AGMR2015,AGS2014,EKS2015} for further reference. We recall some basic properties of $\mathrm{RCD}^*(K,N)$ spaces.

\begin{theorem}(Volume comparison, \cite{BacherSturm2010,CavallettiSturm2012})
Let $K \in \mathbb{R}$ and $N \ge 1$. Then there exists a function $\Lambda_{K,N}(\cdot,\cdot): \mathbb{R}^+ \times \mathbb{R}^+ \mapsto \mathbb{R}^+$ such that if $(X,d,\mathfrak{m})$ is an $\mathrm{RCD}^*(K,N)$ space, we have 
$$\frac{\mathfrak{m}(B_r(x)}{\mathfrak{m}(B_R(x))} \ge \Lambda_{K,N}(r,R), \forall 0 < r < R, x \in X.$$
\end{theorem}

\begin{theorem}(splitting, \cite{Gigli2013})
Let $(X,d,\mathfrak{m})$ be an $\mathrm{RCD}^*(0,N)$ space where $N \ge 1$. Suppose that $X$ contains a line. Then $(X,d,\mathfrak{m})$ is isomorphic to $(X' \times \mathbb{R},d' \times d_E ,\mathfrak{m}' \times \mathcal{L}_1)$ where $d_E$ is the Euclidean distance, $\mathcal{L}_1$ is the Lebesgue measure and $(X',d',\mathfrak{m}')$ is an $\mathrm{RCD}^*(0,N-1)$ space if $N \ge 2$ or a singleton if $N < 2$.
\end{theorem}

We call a point $x$ regular if $(r_iX, x)$ converges to $\mathbb{R}^k$ for any sequence $r_i \to \infty$, where $k$ is an integer between $1$ and $N$. 

\begin{theorem}(Regular points have full measure, \cite{GMR2015,MondinoNaber2019})
Assume that $(X,d,\mathfrak{m})$ is an $\mathrm{RCD}^*(K,N)$ space where $K \in \mathbb{R}$ and $N \ge 1$. Then $\mathfrak{m}$-a.e. $x \in X$ is a regular point.
\end{theorem}
\subsection{$\delta$-cover}

Let's recall the $\delta$-cover introduced in \cite{SormaniWei2004}; see also \cite{MondinoWei2019}.

Given an open covering $\mathcal{U}$ of a length metric space $X$, there is a covering space $\tilde{X}_{\mathcal{U}}$ such that $\pi_1(\tilde{X}_{\mathcal{U}}, \tilde{p})$ is isomorphic to $\pi_1(X,\mathcal{U},p)$, where $p=\pi(\tilde{p})$ and $\pi_1(X,\mathcal{U},p)$ is the normal subgroup generated by all $[\alpha^{-1} \circ \beta \circ \alpha] \in \pi_1(X,p)$; $\beta$ is a loop lying in an element of $\mathcal{U}$ and  $\alpha$ is a path from $p$ to $\beta(0)$.
\begin{defn} (relative $\delta$-cover)
Let $(X,d)$ be a length metric space. The $\delta$-cover of $X$, denoted by $\tilde{X}^{\delta}$, is defined to be $\tilde{X}_{\mathcal{U}_{\delta}}$, where $\mathcal{U}_{\delta}$ is an open covering consisting of all $\delta$ balls in $X$. \\
For any $0 <r < R$ and $x \in X$, let $\tilde{B}(x,R)^{\delta}$ be the $\delta$ cover of $B_R(x)$. A connected component of $\pi^{-1}(B_r(x)) \subset \tilde{B}(x,R)^{\delta}$ is called a relative $\delta$-cover, denoted by $\tilde{B}(x,r,R)^{\delta}$.
\end{defn}

Mondino and Wei proved the stability of relative $\delta$-cover (Theorem 4.5 in \cite{MondinoWei2019}). We slightly change the form for our purpose.
\begin{theorem}\label{stability}(\cite{MondinoWei2019})
Let $(X,d,\mathfrak{m})$ be an $\mathrm{RCD}^{*}(K,N)$ space for some $K \in \mathbb{R}$ and $N \in (1, \infty)$ and $x \in X$. For any $R >0$, there exists $\delta_0$ so that $\tilde{B}(x,R/10,R)^{\delta}$ are all same for any $\delta \le \delta_0$.
\end{theorem}  

The idea of the proof of Theorem \ref{stability} is to first prove the case of regular points using Halfway Lemma and  Abresch-Gromoll inequality, then use the fact that regular points are dense and volume comparison theorem.
\section{Proof of the Main Theorem}

Recall in \cite{Wang2021}, a key lemma says that for any loop $\gamma$ in a small neighborhood of a Ricci limit space $Y$, we can find a loop $\gamma_i$ in $M_i$ so that $\gamma_i$ is point-wise close to $\gamma$ and homotpoic to a very short loop by the controlled homotopy image. 

A similar (and easier) idea works for an $\mathrm{RCD}^*(K,N)$ space $X$. Using Theorem \ref{stability}, we can show that any loop $\gamma$, in a small neighborhood of an $\mathrm{RCD}^*(K,N)$ space, is homotopic to some loops in very small balls by a controlled homotopy image. Then we can use same construction to find a homotopy map.

\begin{lemma}\label{lemma3}
Fix $x \in \bar{B}_{1/2}(p)$ in an $\mathrm{RCD}^*(K,N)$ space $(X,p)$. For any $l<1/2$ and small $\delta>0$, there exists $r<l$ and $k \in \mathbb{N}$ so that any loop $\gamma \subset B_r(x)$ is homotopic to the union of some loops $\gamma_j$ ($1 \le j \le k$) in $\delta$-balls and the homotopy image is in $B_{40l}(x)$. 
\\ 
To be precise, in the unit disc $D \subset \mathbb{R}^2$, we can find $k$ disjoint discs $D_0,D_2,...,D_k$ with radius less than $1/10$ and away from the boundary of $D$. Then there is a continuous map $H$ from the completion of $D-\sum_{j=1}^k D_j$ to $B_{40l}(x)$ so that $H(\partial D)=\gamma$ and $H(\partial D_j)=\gamma_j$ where each $\gamma_j$ is contained in a $\delta$-ball.
\end{lemma}

\begin{proof}
We may assume $\delta$ is small enough to apply Theorem \ref{stability} with $\tilde{B}(x,4l,40l)^{\delta}$. Let $\tilde{x}$ be a pre-image of $x$ in the stable relative $\delta$-cover. We may assume $r$ small enough so that $B_r(x)$ is isometric to $B_r(\tilde{x})$. Given a loop $\gamma$ in $B_r(x)$, we can lift $\gamma$ to a loop $\tilde{\gamma}$ in $B_r(\tilde{x})$.

Since $\tilde{\gamma}$ is a loop in the $\delta$-cover of $B_{40l}(x)$, it's homotopic, by $\tilde{H}$, to 
$$\tilde{c}_1\tilde{\gamma}_1\tilde{c}_1^{-1}...\tilde{c}_k\tilde{\gamma}_{k}\tilde{c}_k^{-1}$$ 
where $\tilde{c}_j$ is a path from $\tilde{\gamma}(0)$ to a point $\tilde{x}_j \in \tilde{B}(x,40l)^{\delta}$ and $\tilde{\gamma}_j$ is a loop contained in $B_{\delta}(\tilde{x}_j)$, $0 \le j \le k$. 

Let 
$$c_j=\pi(\tilde{c}_j),\gamma_j=\pi(\tilde{\gamma}_j), 0 \le j \le k, $$
and
$$\gamma^{\delta} =c_1\gamma_1c_1^{-1}c_1\gamma_1c_1^{-1}...c_k\gamma_{k}c_k^{-1}.$$
where $\pi: \tilde{B}(x,40l)^{\delta} \to B_{40l}(x)$ is the projection map. So $H=\pi(\tilde{H})$ is a homotopy map between $\gamma$ and $\gamma^{\delta}$. Recall that here $H$ is a continuous map from $[0,1] \times [0,1]$ to  $X$ so that $H(0,t)=H(1,t)=p$, $H(t,0) = \gamma(t)$, $H(t,1) = \gamma^{\delta}(t)$. We shall modify the definition area of $H$ by gluing. We first glue $(0,t)$ and $(1,t)$ for each $t$, and get $H$ mapping from an annulus to $B_{40l}(x)$ so that the image of outer boundary is $\gamma$ and the image of inner boundary is $\gamma^\delta$. Then we glue the parts on the inner boundary corresponding to $c_j$ and $c_j^{-1}$ for each $j$. Now the definition area of $H$ is (up to a homeomorphism) the completion of $D-\sum_{j=1}^k D_j$ to $B_{40l}(x)$, where $D_j$ is a disc contained in $D$; $H(\partial D)=\gamma$ and $H(\partial D_j)=\gamma_j$.
\end{proof}

The proof of the Main Theorem is an inductive argument using Lemma \ref{lemma3}.
\begin{proof}[Proof of the main theorem]
Fix $R>0$, choose $l_i= R/100^i$ and take $l=l_i$ in Lemma \ref{lemma3}, $i=1,2,...$. Although the choice of $r$ in Lemma \ref{lemma3} depends on $x$, since $X$ is locally compact, we can find $r_i$ working for all $x \in \bar{B}_{1/2}(p)$ and $l=l_i$ in Lemma \ref{lemma3}. Choose $\delta_i < r_{i+1}$. For all $x \in \bar{B}_{1/2}(p)$, any loop in $B_{r_i}(x)$ is homotopic to loops in $\delta_i$-balls and the homotopy image is contained in $B_{l_i}(x)$. We shall show any loop $\gamma$ in $B_{r_1}(p)$ is contractible in $B_{R}(p)$.

Roughly speaking, we first shrink $\gamma$ to loops in $\delta_1$-balls, the second step is to shrink each new loop to smaller loops in $\delta_2$-balls, etc. Since the homotopy to shrink each loop is contained in a $l_i$-ball in the $i$-th step, this process converges to a homotopy map which contracts $\gamma$ while the image is contained in a ball with radius $\sum_{i=1}^{\infty} l_i < R$. We show a detailed proof in the following. 

Since $\gamma$ is in $B_{r_1}(p)$, we can apply Lemma \ref{lemma3} with $\delta=\delta_1$ and $l=l_1$. There are $k_1$ loops $\gamma_j^1 \subset B_{\delta_1}(z_j^1)$ where $z_j^1 \in X$, $k_1$ disjoint balls $D_j^1 \subset D$ with radius less than $1/10$, and a continuous map $H_1$ from the completion $D-\sum_{j=1}^{k_1} D_j^1$ to $B_{l_1}(p)$, $H_1(\partial D)=\gamma$, $H_1(\partial D_j^1)=\gamma_j^1$, $1 \le j \le k_1$.

If $k_1 =0$, $\gamma$ is homotopic to a point by $H_1$. We may assume $k_1 >0$. We shall shrink $\gamma_j^1$ in the second step. First consider $\gamma_1^1$ for example. Note that $\delta_1 < r_2$, we can apply lemma \ref{lemma3} to $\gamma_1^1$ with $l=l_2$ and $\delta=\delta_2$. There are $k_{2,1}$ (here the subscript $(2,1)$ means the first loop in the second step) loops $\gamma_j^{2,1} \subset B_{\delta_1}(z_j^{2,1})$, where $z_j^{2,1} \in X$ and $j \le k_{2,1}$. There are $k_{2,1}$ disjoint balls $D_j^{2,1} \subset D_1^1$ with radius less than $1/100$ and we can find a continuous map $H_{2,1}$ from the completion $D_1^1-\sum_{j=1}^{k_{2,1}} D_j^{2,1}$ to $B_{l_2}(z_1^1)$, so that $H_{2,1}(\partial D_1^1)=\gamma_1^1$ and $H_{2,1}(\partial D_j^{2,1})=\gamma_j^1$, $1 \le j \le k_{2,1}$. Note that $H_1(\partial D_1^1)=H_{2,1}(\partial D_1^1)=\gamma_1^1$, we can extend $H_1$ using $H_{2,1}$. 

Repeat the above process for each $1 \le j \le k_1$ and extend $H_1$ to a new continuous map $H_2$. That is, we can find $k_2$ loops $\gamma_j^2 \subset B_{\delta_2}(z_j^2)$ where $z_j^2 \in X$, $k_2$ disjoint balls $D_j^2 \subset D$ with radius less than $1/100$, each $D_j^2$ is contained in one of $D_{j'}^1$, $1 \le j' \le k_1$, and $H_2$ from the completion $D-\sum_{j=1}^{k_2} D_j^1$ to $B_{l_1+l_2}(p)$, $H_2(\partial D)=\gamma$, $H_2(\partial D_j^2)=\gamma_j^2$, $1 \le j \le k_2$. Moreover $H_2$ coincides with $H_1$ wherever $H_1$ is defined.

By induction, in the $i$-th step, we can find $k_i$ loops $\gamma_j^i \subset B_{\delta_i}(z_j^i)$ where $z_j^i \in X$, $k_i$ disjoint balls $D_j^i \subset D$ with radius less than $1/10^i$, each $D_j^i$ is contained in one of $D_{j'}^{i-1}$, $1 \le j' \le k_{i-1}$, and $H_i$ from the completion $D-\sum_{j=1}^{k_i} D_j^i$ to $X$, $H_i(\partial D)=\gamma$, $H_i(\partial D_j^i)=\gamma_j^i$. Moreover $H_i=H_{i-1}$ where $H_{i-1}$ is defined. 

We prove that $H_i$ converges to a continuous map $H$ from $D$ to $B_R(x)$ with $H(\partial D)=\gamma$; thus $\gamma$ is contractible in $B_R(x)$. For any $q \in D$, if $H_i(q)$ is defined for some $i$ (thus for all large $i$), define $H(q)=H_i(q)$. Then we have $H(\partial D)=\gamma$. In this case, $H$ is continuous at $q$ because  we may assume $q$ is not on $\partial D_j^i$ by taking large $i$ (in lemma \ref{lemma3}, smaller discs $D_j$ are away from $\partial D$). Otherwise we assume $H_i$ is not defined at $q$ for all $i$, then for each $i$ there is a disc $D_{j_i}^i$ so that $q \in D_{j_i}^i$. Moreover, $D_{j_i}^i \subset D_{j_{i-1}}^{i-1}$ for each $i$. Let $S_i \subset D_{j_i}^i$ be the set where $H_i$ is defined.  Recall that the image of $H_i(S_i)$ is contained in a $l_i$-ball, thus  $H_i(S_i)$ converges to a point as $i \to \infty$; define $H(q)$ to be this point. The image  $H(D_{j_i}^i)$ is contained in a ball with radius $\sum_{j=i}^{\infty} l_j$ which converges to $0$ as $i \to \infty$, thus $H$ is continuous at $q$. Take $i=1$, $H(D)$ is contained in a ball with radius $\sum_{j=1}^{\infty} l_j < R$.
\end{proof}

\begin{remark}
We actually proved that for a locally compact metric space, if any local relative $\delta$-cover is stable, then  it is semi-locally simply connected.  

We should also mention that harmonic archipelago is an example that there is a loop which can be homotopic to a loop in an arbitrarily small ball, but is not contractible \cite{Fabel}.
\end{remark}
\bibliographystyle{plain} 
\bibliography{semii}

\begin{thebibliography}{10}

\bibitem{AGMR2015}
L.~Ambrosio, N.~Gigli, A.~Mondino, and T.~Rajala.
\newblock Riemannian {R}icci curvature lower bounds in metric measure spaces
  with σ-finite measure.
\newblock {\em Trans. Amer. Math Soc}, 2015.

\bibitem{AGS2014}
L.~Ambrosio, N.~Gigli, and G.~Savare.
\newblock Metric measure spaces with {R}iemannian {R}icci curvature bounded
  from below.
\newblock {\em Duke Math J.}, 2014.

\bibitem{BacherSturm2010}
K.~Bacher and K.-T. Sturm.
\newblock Localization and tensorization properties of the curvature-dimension
  condition for metric measure spaces.
\newblock {\em Journal of Functional Analysis}, 2010.

\bibitem{CavallettiSturm2012}
F.~Cavalletti and K.-T. Sturm.
\newblock Local curvature-dimension condition implies measure-contraction
  property.
\newblock {\em Journal of Functional Analysis}, 2012.

\bibitem{CheegerColding1997}
Jeff Cheeger and Tobias~H. Colding.
\newblock On the structure of spaces with {R}icci curvature bounded below. i.
\newblock {\em J. Differential Geom.}, 46(3):406--480, 1997.

\bibitem{CheegerColding2000a}
Jeff Cheeger and Tobias~H. Colding.
\newblock On the structure of spaces with {R}icci curvature bounded below. ii.
\newblock {\em J. Differential Geom.}, 54(1):13--35, 2000.

\bibitem{CheegerColding2000b}
Jeff Cheeger and Tobias~H. Colding.
\newblock On the structure of spaces with {R}icci curvature bounded below. iii.
\newblock {\em J. Differential Geom.}, 54(1):37--74, 2000.

\bibitem{CheegerNaber2013}
Jeff Cheeger and Aaron Naber.
\newblock Lower bounds on {R}icci curvature and quantitative behavior of
  singular sets.
\newblock {\em Invent. Math.}, 191(2):321--339, 2013.

\bibitem{CheegerNaber2015}
Jeff Cheeger and Aaron Naber.
\newblock Regularity of {E}instein manifolds and the codimension 4 conjecture.
\newblock {\em Ann. of Math. (2)}, 182(3):1093--1165, 2015.

\bibitem{ColdingNaber2012}
Tobias~H. Colding and Aaron Naber.
\newblock {Sharp Hölder continuity of tangent cones for spaces with a lower
  Ricci curvature bound and applications}.
\newblock {\em Annals of Mathematics}, 176(2):1173--1229, 2012.

\bibitem{EKS2015}
M.~Erbar, K.~Kuwada, and K.-T. Sturm.
\newblock On the equivalence of the entropic curvature-dimension condition and
  {B}ochner’s inequality on metric measure spaces.
\newblock {\em Invent. Math}, 2015.

\bibitem{Fabel}
Paul Fabel.
\newblock The fundamental group of the harmonic archipelago.
\newblock {\em arXiv:math/0501426}, 2005.

\bibitem{Gigli2013}
N.~Gigli.
\newblock The splitting theorem in non-smooth context.
\newblock {\em Submitted paper, arXiv:1302.5555}, 2013.

\bibitem{GMR2015}
N.~Gigli, A.~Mondino, and T.~Rajala.
\newblock Euclidean spaces as weak tangents of infinitesimally {H}ilbertian
  metric spaces with {R}icci curvature bounded below.
\newblock {\em Journal fur die Reine und Angew. Math.}, 2015.

\bibitem{MondinoNaber2019}
A.~Mondino and A.~Naber.
\newblock Structure theory of metric-measure spaces with lower {R}icci
  curvature bounds.
\newblock {\em J. Eur. Math.}, 2019.

\bibitem{MondinoWei2019}
Andrea Mondino and Guofang Wei.
\newblock On the universal cover and the fundamental group of an {RCD∗(K,
  N)-space}.
\newblock {\em J. Reine Angew. Math.}, 2019.

\bibitem{PanWang2021}
Jiayin Pan and Jikang Wang.
\newblock Some topological results of {R}icci limit spaces.
\newblock {\em To appear in Transactions of the American Mathematical Society},
  2021.

\bibitem{PanWei2019}
Jiayin {Pan} and Guofang {Wei}.
\newblock {Semi-local simple connectedness of non-collapsing Ricci limit
  spaces}.
\newblock {\em To appear in Journal of the European Mathematical Society,
  arXiv:1904.06877}, 2019.

\bibitem{SantosZamora}
Jaime Santos-Rodriguez and Sergio Zamora.
\newblock On fundamental groups of {RCD} spaces.
\newblock {\em arXiv:2210.07275}, 2022.

\bibitem{SormaniWei2004}
Christina Sormani and Guofang Wei.
\newblock Universal covers for {H}ausdorff limits of noncompact spaces.
\newblock {\em Transactions of the American Mathematical Society},
  356(3):1233--1270, 2004.

\bibitem{Tuschman1995}
W.~Tuschman.
\newblock Hausdorff convergence and the fundamental group.
\newblock {\em Math. Z., 208}, 1995.

\bibitem{Wang2021}
Jikang Wang.
\newblock Ricci limit spaces are semi-locally simply connected.
\newblock {\em arXiv:2104.02460}, 2021.

\end{thebibliography}
\end{document}